\documentclass{article}
\usepackage{setspace}

\usepackage{amsmath,amsthm,amssymb}
\usepackage[all]{xy}
\usepackage{stmaryrd}

\newcommand{\cat}{\mathcal{C}}
\newcommand{\Cat}{\mathrm{C}}
\newcommand{\Z}{\mathbb{Z}}
\newcommand{\K}{\mathbb{K}}
\newcommand{\N}{\mathbb{N}}

\newcommand{\Ob}{\operatorname{Ob}}

\newcommand{\Ker}{\operatorname{Ker}}

\newcommand{\Supp}{\operatorname{Supp}}
\newcommand{\GL}{\operatorname{GL}}
\newcommand{\Vect}{\mathcal{V}}
\newtheorem{theorem}{Theorem}[section]
\newtheorem{proposition}[theorem]{Proposition}

\newtheorem{lemma}[theorem]{Lemma}
\theoremstyle{definition}
\newtheorem{definition}{Definition}[section]
\theoremstyle{remark}
\newtheorem{remark}{Remark}[section]

\begin{document}
\title{Perfect category-graded algebras}
\author{Ana Paula Santana and Ivan Yudin
\thanks{Financial support by CMUC/FCT gratefully acknowledged by both
authors.}
\thanks{The second author's work is supported by the FCT Grant
SFRH/BPD/31788/2006.}}

\maketitle

\begin{abstract}
In a perfect category every object has a minimal projective resolution. We give
a criterion for the category of modules over a category-graded
algebra to be perfect.
\end{abstract}

\section{Introduction}
In \cite{semi-perfect} the second author explored
homological properties of algebras graded over a small category. Our interest in
these algebras arose from our research on the homological properties of Schur algebras, but we
believe that they play an important organizational role in representation
theory in general. 

Recall that an abelian category $\cat$ is called \emph{perfect} if every object
of $\cat$ has a \emph{projective cover} (see Section~\ref{preliminaries}). The
existence of projective covers for every object guarantees the existence of
minimal projective resolutions for every object in the category. 
The category $\cat$ is called \emph{semi-perfect} if every finitely generated
object has a projective cover. We say that a category-graded algebra $A$ is
(semi)-perfect if the category of $A$-modules is (semi)-perfect. 
In~\cite{semi-perfect} it was given a criterion for  category-graded algebras
to be semi-perfect. This criterion is sufficient to ensure that all
category-graded algebras which appear in~\cite{main} are semi-perfect. But this
is not enough to prove the existence of a minimal  projective
resolution for some of them, as the kernel of a projective cover may
not be finitely generated. In this article we fill this gap by giving a criterion
for a category-graded algebra to be perfect and extend the results
of~\cite{semi-perfect} to algebras over an arbitrary commutative ring with
identity.

Next we introduce the notions related with category-graded algebras that will
be needed and explain the main result in
more detail. 
Let $R$ be a commutative ring with identity. We will write $\otimes$ for the tensor product of
two $R$-modules over $R$.  Given a small category $\Cat$, 
a \emph{$\Cat$-graded $R$-algebra} (see~\cite{semi-perfect})  is a
collection of $R$-modules $A_{\alpha}$ parametrised by the arrows
$\alpha$ of $\Cat$, with preferred elements $e_s\in A_{1_s}$ for every
object $s$ of $\Cat$ and a collection of $R$-module homomorphisms $\mu_{\alpha,\beta}\colon
A_{\alpha}\otimes A_{\beta}\to A_{\alpha\beta}$ for every composable
pair of morphisms $\alpha$, $\beta$ of $\Cat$. For $a\in A_{\alpha}$ and
$b\in A_{\beta}$ we shall write $ab$ for $\mu_{\alpha,\beta}(a\otimes b)$.  For every composable
triple $\alpha$, $\beta$, and $\gamma$ of arrows in $\Cat$ and $a\in
A_{\alpha}$, $b\in A_{\beta}$, and $c\in A_{\gamma}$ 
we require  associativity
$$
a(bc) = (ab)c.
$$
Suppose also that $\alpha\colon s\to t$. Then we require
$$
e_ta = a = ae_s,\  \mbox{    for any $\alpha\in A_\alpha$}. 
$$

A \emph{$\Cat$-graded module} $M$ over a $\Cat$-graded $R$-algebra $A$ is a
collection of $R$-modules 
$M_{\gamma}$ parametrised by the arrows $\gamma$ of $\Cat$ with 
$R$-module homomorphisms
$r_{\alpha,\beta}\colon A_{\alpha}\otimes M_{\beta}\to
M_{\alpha\beta}$ for every composable pair of morphisms in $\Cat$.
We shall write $am$ instead of $r_{\alpha,\beta}(a\otimes m)$ for $a\in
A_{\alpha}$ and $m\in M_{\beta}$.

As always we will assume the usual module axioms:
\begin{align*}
	a(bm) = (ab)m & \mbox{ for $a\in A_{\alpha}$, $b\in A_{\beta}$, $m\in
	M_{\gamma}$},
\end{align*}
	where $\alpha$, $\beta$, and $\gamma$ are composable; and
	\begin{align*}
		e_t m = m,
\end{align*}
where $\gamma\colon s\to t$ and $m\in M_{\gamma}$. 

An \emph{$A$-homomorphism} between two $\Cat$-graded $A$-modules $M$ and
$N$ is a collection of $R$-module homomorphisms $f_{\gamma}\colon M_{\gamma}\to
N_{\gamma}$ such that for every composable pair of morphisms
$\alpha$, $\beta\in \Cat$ 
$$
f_{\alpha\beta}(am) = a f_{\beta}(m).
$$

We denote the category of all $\Cat$-graded
$A$-modules by $A$-mod.

For morphisms $\beta\colon s\to u$ and $\gamma\colon s\to t$ in $\Cat$ define
$$
A\left( \gamma:\beta \right) = \bigoplus_{\alpha\colon \alpha\beta=\gamma}
A_{\alpha}. 
$$
Note that $A\left( \gamma:\gamma \right)$ is a ring with  unit
$e_{t}$  and the multiplication induced by the maps
$\mu_{\alpha,\beta}$.

The main result of this paper is 

\begin{theorem}
	\label{main0}
	Let $\Cat$ be a small category and  $A$  a $\Cat$-graded
	$R$-algebra. Suppose that  every sequence  $\left( \alpha_k \right)_{k\in
	\N}$ of arrows in $\Cat$ satisfying for all $k, l\ge 1$  $$A\left( \alpha_k:\alpha_{k+l}
	\right)\not\cong 0$$  has at least one element repeated	infinitely many times.  Then $A$-mod is a
	perfect category if and only if the rings $A\left( \gamma:\gamma
	\right)$ are left perfect for all maps $\gamma$ in $\Cat$. 
\end{theorem} 
Let us write $\gamma\succ\beta$ if $\beta$ is a right divisor of $\gamma$. Then
as a corollary we get
\begin{theorem}
	\label{main1}
	Let $\Cat$ be a small category such that every sequence $\beta_1\succ
	\beta_2\succ\dots$ of morphisms in $\Cat$ has at least one element
	repeated
	infinitely many times. 
	 Then for a $\Cat$-graded $R$-algebra $A$ the category $A$-mod is 
	perfect if and only if the rings $A\left( \gamma:\gamma \right)$ are
	left perfect for all arrows $\gamma\in \Cat$.
\end{theorem}

The main idea of the proof of Theorem~\ref{main0} is to apply the general
criterion of perfectness obtained in~\cite{harada}. Therefore we start in
Section~\ref{preliminaries} with a result on the radical of an abelian category
and a
recollection of notions used in that work.  Section~\ref{harada} is devoted to
Harada's criterion and the study of perfectness of a class of abelian
categories, which will be useful in the sequel. In Section~\ref{mains} we  prove  the main result and in
Section~\ref{examples} we give  examples and indicate connections with 
previously known results.  

Throughout this article $R$ denotes a commutative ring with identity. 
For undefined notation the reader is referred to~\cite{semi-perfect}.

\section{Preliminaries}
\label{preliminaries}

The notion of radical for general additive categories was introduced
in~\cite{sai}. 
Let $\cat$ be an additive category. An \emph{ideal} $I$ of $\cat$ is a
collection of subgroups $I(A,B)$ of $\cat(A,B)$ for each $A$, $B\in \Ob\cat$,
such that 

\begin{align*}
	I(B,C) \cat(A,B) &\subset I(A,C)\\
	\cat(B,C) I(A,B) & \subset I(A,C).
\end{align*}

\begin{definition}[\cite{sai}]
	\label{dar}
	A \emph{radical} of an additive category 
	$\cat$ is an ideal $I$ of
	$\cat$ such that for every object $A$ of $\cat$ we have $I(A,A) = J(\cat(A,A))$, where $J$ denotes the Jacobson
	radical of the ring.
\end{definition}
	Let $\cat$ be an abelian category. 
Then $\cat$  has a unique radical. This fact was used
without explicit proof in~\cite{harada}. For completeness we provide a proof in
the appendix. 
We will also write $J$ for the radical of $\cat$. 

Given two objects  $A$, $B$ of $\cat$
	we will denote by $\pi_A\colon A\oplus
	B\to A$, $\pi_B\colon A\oplus B\to B$, $i_A\colon A\to A\oplus B$, and
	$i_B\colon B\to A\oplus B$ the canonical projections and inclusions
	associated with the definition of the direct sum $A\oplus B$. We will need the following technical property of the radical of $\cat$.
\begin{proposition}
	\label{radical}
	Let $\cat$ be an abelian category and $A$, $B$ objects of $\cat$.
	Suppose that $A= A'\oplus A''$ and $B=B'\oplus B''$. 
	 Then $J(A',B') = \pi_{B'} J(A,B) i_{A'}$.
 \end{proposition}
\begin{proof}
	Since $J$ is an ideal we have $\pi_{B'} J\left( A,B \right)i_{A'}\subset
	J\left( A',B' \right)$. Also 
	$$
	J\left( A',B' \right) = 1_{B'} J\left( A',B' \right) 1_{A'} =
	\pi_{B'}i_{B'} J\left( A',B' \right) \pi_{A'} i_{A'} \subset
	\pi_{B'}J\left( A,B \right) i_{A'}
	$$
	and the desired equality follows. 
\end{proof} 

Next we introduce some standard notation which will be used in the following
sections. 

We say that $X\subset Y$ is \emph{a small subobject} of $Y$ if for any $S\subset Y$
such that $X+S= Y$ we have $S=Y$. An epimorphism $\pi\colon P\twoheadrightarrow
Y$, where $P$ is projective, is  called a \emph{projective cover} of Y whenever
$\Ker \pi$ is a small subobject of $P$. 

Note that in a perfect abelian category every object has a (unique up to
isomorphism) \emph{minimal projective resolution}. By definition a minimal
projective resolution of an object $X$ is an exact complex
$(P_{\bullet},d_{\bullet})$ with a
map $\varepsilon\colon P_0\twoheadrightarrow X$, such that 
the maps $d_k\colon P_{k+1} \to \Ker (d_{k-1})$ and $\varepsilon$ are  projective
covers. The existence
of minimal projective resolutions in a perfect category can be shown by
induction. 

\section{Harada criterion}
\label{harada}
In this section we give a sufficient and a necessary condition for a Grothendieck category
$\cat$ to be perfect. These are based on Harada's criterion of perfectness,
Corollary~1~p.338~of~\cite{harada}. 
The crucial ingredient of  this criterion
is the
notion of  $T$-nilpotent system. 
\begin{definition} A set of objects $\left\{\, M_i \,\middle|\,  i\in I \right\}$ in an
	abelian category is called
	\emph{$T$-nilpotent}  if for any sequence of maps
	$f_k\in J\left( M_{i_k}, M_{i_{k+1}} \right)$, $k=1,2,\dots$, and every small subobject
	$X$ of $M_{i_1}$ there is a natural number $m$ such that $f_mf_{m-1}\dots
	f_1(X)=0$.
\end{definition} 

\begin{definition}
	Let $\cat$ be an abelian category. 
	We say that an object $B\in \cat$ is \emph{semi-perfect}
	(\emph{completely indecomposable}) if the ring $\cat(B,B)$ is
	semi-perfect (local). 
\end{definition}
Note that our definition of semi-perfect object is different from the definition
given in~\cite{harada} on p.~330, but this does not interfere with the work.  

Let $\left\{\, P_\alpha \,\middle|\,  \alpha\in I \right\}$ \label{construction} be a generating set
of semi-perfect objects of an abelian  category $\cat$. Then each ring
$\cat(P_\alpha,P_\alpha)$ is semi-perfect. By Theorem~27.6
of~\cite{anderson-fuller}, for each $\alpha$ the ring
$\cat(P_{\alpha},P_{\alpha})$ has a complete set of orthogonal idempotents
$e_{\alpha,1}$, $e_{\alpha,2}$, \dots, $e_{\alpha,n_{\alpha}}$ and for every
$\alpha\in I$ and every $1\le j \le n_{\alpha}$ the ring  
$e_{\alpha,j}\cat(P_{\alpha}, P_{\alpha})e_{\alpha,j}$  is local. We denote by
$P_{\alpha,j}$ the direct summand of $P_\alpha$ that corresponds to
$e_{\alpha,j}$. We also write $\pi_{\alpha,j}$ for the canonical projection of
$P_{\alpha}$ on $P_{\alpha,j}$ and $i_{\alpha,j}$ for the canonical embedding of
$P_{\alpha,j}$ in $P_{\alpha}$. 
\begin{proposition}
	\label{indecomposable}
	The objects $P_{\alpha,j}$ are completely indecomposable.
\end{proposition}
\begin{proof}
	Since $e_{\alpha,1}$, \dots, $e_{\alpha,n_{\alpha}}$ is a complete
	orthogonal set of idempotents the ring $\cat(P_{\alpha,j},
	P_{\alpha,j})\cong e_{\alpha,j} \cat(P_{\alpha},
	P_{\alpha})e_{\alpha,j}$ is local. Thus $P_{\alpha,j}$ is
	completely indecomposable.
\end{proof}

\begin{proposition}
	\label{new}
	Let $\cat$ be a Grothendieck category with a generating set of finitely
	generated objects. Suppose $\cat$ has a generating set $\left\{\,
	P_{\alpha}
	\,\middle|\, \alpha\in I \right\}$ of semi-perfect projective objects.
	If $\left\{\, P_{\alpha} \,\middle|\, \alpha\in I \right\}$ is a 
	$T$-nilpotent system then $\cat$ is perfect. 
\end{proposition}
\begin{proof}
	In this proof we are going to apply
	Corollary~1~on~p.338~of~\cite{harada}. This claims
	that if $\cat$ has a generating set of finitely generated objects and
	$\left\{\, Q_{\beta} \,\middle|\, \beta\in K \right\}$ is a
	$T$-nilpotent generating set of completely indecomposable projective
	objects, then
	$\cat$ is perfect. Thus we have to construct a $T$-nilpotent generating
	set of completely indecomposable projective objects. 
	
	If we apply the construction described above to $\left\{\, P_{\alpha} \,\middle|\,
	\alpha\in I \right\}$, we get a generating set $\mathcal{G} = \left\{\, P_{\alpha,j}
	\,\middle|\, \alpha\in I,\ j=1,\dots,n_{\alpha} \right\}$. Every object
	$P_{\alpha,j}$ is a direct summand of $P_{\alpha}$ and so
	$P_{\alpha,j}$ is projective. The object $P_{\alpha,j}$ is also
	completely indecomposable by Proposition~\ref{indecomposable}. 

	Now we will show that $\mathcal{G}$ is $T$-nilpotent. Let
	$P_{\alpha_1,j_1}$, $P_{\alpha_2,j_2}$, \dots be a sequence of objects
	in $\mathcal{G}$ and $f_k\in J(P_{\alpha_k,j_k},
	P_{\alpha_{k+1},j_{k+1}})$. 
	From Proposition~\ref{radical} it follows that $J(P_{\alpha_k,j_k},P_{\alpha_{k+1},j_{k+1}})=
	\pi_{\alpha_{k+1},j_{k+1}} J(P_{\alpha_k},P_{\alpha_{k+1}})
	i_{\alpha_k,j_k}$.
	Thus there is $\widetilde{f}_k\in
	J(P_{\alpha_k},P_{\alpha_{k+1}})$ such that $f_k =
	\pi_{\alpha_{k+1},j_{k+1}}  \widetilde{f}_k i_{\alpha_k,j_k}$. 
	Denote by $g_k$ the element $i_{\alpha_{k+1},j_{k+1}}\pi_{\alpha_{k+1},j_{k+1}}
	\widetilde{f}_k$
	of $J(P_{\alpha_k}, P_{\alpha_{k+1}})$. Then we have
	$$
	f_r\dots f_1 = \pi_{\alpha_{r+1},j_{r+1}} g_r\dots g_1 i_{\alpha_1,j_1}.
	$$
	Let $X$ be a small subobject of $P_{\alpha_1,j_1}$. Then
	$i_{\alpha_1,j_1}(X)$ is a small subobject of $P_{\alpha}$.
	Since $\left\{\, P_{\alpha} \,\middle|\, \alpha\in I \right\}$ is
	$T$-nilpotent 
	there is some $n$ such that 
	$$
	g_n\dots g_1 i_{\alpha_1,j_1}(X) =0.
	$$
	But then $f_n\dots f_1(X) =0$. 
\end{proof}

To prove the next proposition we shall use the following consequence of Axiom of
Choice (see example~1 to Theorem~III.7.4.1 of \cite{bourbaki}).

\begin{lemma}
	\label{AC}
	Let $g_k\colon S_{k+1}\to S_{k}$, $k\in \N$,  be a sequence of maps between finite
	non-empty sets. Then 
	$$
	\varprojlim_{k}S_k := \left\{\, \left( s_k \right)_{k\in \N} \,\middle|\, s_k\in S_k,\
	g_k\left( s_{k+1}  \right) = s_k\right\}
	$$
	is a non-empty set.
\end{lemma}

\begin{proposition}
	\label{harada3}
	Let $\cat$ be a perfect Grothendieck category with a generating set
	$\left\{\, P_{\alpha} \,\middle|\, \alpha\in I \right\}$ of finitely generated
	projective objects. Then $\left\{\, P_{\alpha} \,\middle|\, \alpha\in I
	\right\}$ is $T$-nilpotent. 
\end{proposition}
\begin{proof}
	By Theorem~7.2 of \cite{semi-perfect} every object $P_{\alpha}$ is
	semi-perfect. Let $\left\{ e_{\alpha,1}, \dots, e_{\alpha,n_{\alpha}}
	\right\}$ be a complete set of orthogonal idempotents for $\cat\left(
	P_\alpha,P_\alpha \right)$, $\alpha\in I$. Denote by $P_{\alpha,j}$ the
	direct summand that corresponds to $e_{\alpha,i}$. Then $\left\{\,
	P_{\alpha,j}
	\,\middle|\, \alpha\in I,\ 1\le j\le n_{\alpha} \right\}$ is a generating set of completely indecomposable objects. By
	Corollary~1 of Theorem~4 in \cite{harada} $\left\{\, P_{\alpha,j} \,\middle|\,
	\alpha\in I,\ 1\le j\le n_{\alpha} \right\}$
	is $T$-nilpotent, since $\cat$ is perfect. 

	Now suppose that $\left\{\, P_\alpha \,\middle|\, \alpha\in I \right\}$ is not $T$-nilpotent.
	Then there is a sequence $f_k\in J\left( P_{\alpha_k},
	P_{\alpha_{k+1}} \right)$ and a small subobject $X$ of $P_{\alpha_1}$
	such that for every $m\in \N$
	$$
	f_m \dots f_1\left( X \right)\not= 0
	$$
	which is the same as
	$$
 \sum_{j_1=1}^{n_{\alpha_1}}\dots
	\sum_{j_{m+1}=1}^{n_{\alpha_{m+1}}} e_{\alpha_{m+1},j_{m+1}}f_m
	e_{\alpha_m,j_m} \dots e_{\alpha_2,j_2}f_1e_{\alpha_1,j_1}\left( X
	\right)\not= 0.
	$$
	Denote by $S_{m+1}$ the subset of $\left\{ 1,\dots,n_{\alpha_1}
	\right\}\times \dots\times \left\{ 1,\dots,n_{\alpha_{m+1}} \right\}$ of
	elements $\left( j_1,\dots,j_{m+1} \right)$ such that
	$$
	e_{\alpha_{m+1},j_{m+1}} f_me_{\alpha_m,j_m}\dots e_{\alpha_2,j_2}f_1
	e_{\alpha_1,j_1}\left( X \right)\not=0. 
	$$
	Then $S_m$ are finite non-empty sets for every $m\in \N$, and we have maps
	\begin{align*}
		g_m\colon S_{m+1} &\to S_m\\
		\left( j_1,\dots,j_{m+1} \right)&\mapsto \left( j_1,\dots,j_m
		\right).
	\end{align*}
	From Lemma~\ref{AC} it follows that there is a sequence $\left( l_k
	\right)_{k\in\N}$ such that $\left( l_1,\dots,l_m \right)\in S_m$ for
	every $m$. Define $h_m = \pi_{\alpha_{m+1},l_{m+1}}f_m
	i_{\alpha_m,l_m}$. Then $h_m\in J\left( P_{\alpha_m,l_m},
	P_{\alpha_{m+1},l_{m+1}} \right)$. Moreover, for every $m\in \N$
	$$
	i_{\alpha_{m+1},l_{m+1}}h_m\dots h_1(\pi_{\alpha_1,l_1}X) = e_{\alpha_{m+1},l_{m+1}} f_me_{\alpha_m,l_m}\dots e_{\alpha_2,l_2}f_1
	e_{\alpha_1,l_1}\left( X \right) \not=0.
	$$
	Thus $Y= \pi_{\alpha_1,l_1}\left(X\right)$ is a small subobject of $P_{\alpha_1,l_1}$
	such that for all $m\in \N$
	$$
		h_m \dots h_1\left( Y \right)\not= 0,
	$$
	which contradicts the fact that  $\left\{\, P_{\alpha,j} \,\middle|\,
	\alpha\in I,\ 1\le j\le n_{\alpha}
	\right\}$ is
	$T$-nilpotent. 
\end{proof}
\section{The main result}
\label{mains}
Let $\Cat$ be a small category. We define a \emph{$\Cat$-graded $R$-module} $V$ 
to be a 
collection of $R$-modules $V_{\gamma}$ parametrized by the arrows 
$\gamma\in\Cat$. A morphism from a $\Cat$-graded $R$-module $V$ to a $\Cat$-graded
$R$-module $W$ is a collection of $R$-module homomorphisms $f_\gamma\colon V_\gamma\to
W_\gamma$. We will write $\Vect_\Cat$ for the category of $\Cat$-graded $R$-modules and
$\Vect$ for the category of $R$-modules.  

Next we indicate how the results of~\cite{semi-perfect} can be extended from the
case when $R$ is a field to the case of a general commutative ring. 
Let $A$ be a $\Cat$-graded $R$-algebra with multiplication map $\mu$  and $V$ a $\Cat$-graded $R$-module.
Consider the
functor 
$$
F_A\colon \Vect_\Cat\to A\mbox{-mod}
$$
given on objects by the formula 
$$
F_A\left( V \right)_{\gamma} = \bigoplus_{\gamma=\alpha\beta} A_\alpha\otimes
V_\beta
$$
with structure maps $r_{\delta,\gamma}\colon A_{\delta}\otimes F_A\left(
V \right)_{\gamma} \to F_A\left( V \right)_{\delta\gamma}$ defined by the
requirement that its restriction to the component $A_{\delta}\otimes
A_{\alpha} \otimes V_{\beta}$ is $\mu_{\delta,\alpha}\otimes V_{\beta}$. 
On morphisms $F_A$ is defined by requirement that the restriction of $F_A\left( f
\right)_{\gamma}$ to $A_{\alpha}\otimes V_\beta$ is
$A_{\alpha}\otimes f_{\beta}$ for $\alpha$, $\beta$ such that $\alpha\beta=\gamma$. 

Repeating the proof of \cite[Proposition~2.1]{semi-perfect} we get
	that the functor $F_A$ is a left adjoint to the forgetful functor $U\colon
	A\mbox{-mod}\to \Vect_\Cat$. 
The counit $\varepsilon$ of this adjunction is given by the structure
maps of $A$-modules. Namely, if $M$ is an $A$-module with structure maps $r_{\alpha,\beta}$ then the $\left( \alpha,\beta \right)$
component of $\varepsilon_{\gamma}\colon F_A\left( M \right)_{\gamma} \to
M_\gamma$ is $r_{\alpha,\beta}$. 
From the existence of local units it follows that the maps
$\varepsilon_{\gamma}$ are surjective for all $\gamma\in \Cat$.

The proofs of Propositions~3.1,~4.2, and~4.3
of~\cite{semi-perfect} can be
extended without any changes to the case of general $R$. As a consequence we get 

\begin{proposition}
	\label{abelian} Let $A$ be a $\Cat$-graded $R$-algebra. Then the
	category $A$-mod is Grothendieck. In particular, $A$-mod is a complete and
cocomplete abelian category.\end{proposition}

We say that an object $V\in \Vect_\Cat$ is \emph{free} if every component
$V_{\gamma}$ of $V$ is a free $R$-module. 
It is clear that every free $\Cat$-graded $R$-module is projective, as the
lifting condition must be verified componentwise.

Given a $\Cat$-graded $R$-algebra $A$, we say that an
$A$-module $M$ is \emph{free} if there is a free $\Cat$-graded $R$-module
$V$ such that $F_A\left( V \right)\cong M$ in $A$-mod. 
Now we have an analog of \cite[Proposition~5.1]{semi-perfect}.

\begin{proposition}
	\label{freeprojective}
	Let $A$ be a $\Cat$-graded $R$-algebra and $M$ a free $A$-module. Then
	$M$ is projective.
\end{proposition}
For each arrow $\gamma$ in $\Cat$, we  define the $\Cat$-graded $R$-module $R\left[ \gamma \right]$ by
$$
R\left[ \gamma \right]_{\alpha} = 
\begin{cases}
	R & \mbox{if $\alpha=\gamma$}\\
	0 & \mbox{otherwise.}
\end{cases}
$$
Denote $F_A\left( R\left[ \gamma \right] \right)$ by $A\left[ \gamma \right]$. 
\begin{proposition}
	\label{genset}
	The set $\left\{\, A\left[ \gamma \right] \,\middle|\,
	\gamma\in\Cat\right\}$ is a generating set of $A$-mod, whose elements
	are  finitely generated
	projective $A$-modules. 
\end{proposition}
\begin{proof}
	From Proposition~\ref{freeprojective} we know that the objects $A\left[
	\gamma \right]$, $\gamma\in\Cat$,  are projective. By the reasoning on p.105 of~\cite{mitchell} a
	projective object is finitely generated if and only if it is small. To
	check that $A\left[ \gamma \right]$ is small we have
	to show that for any family of $A$-modules $\left\{\, M_i \,\middle|\,
	i\in I \right\}$ and every map of $A$-modules $f\colon A\left[ \gamma \right]\to \bigoplus_{i\in I} M_i
	$ there is a finite subset $I'$ of $I$ such that $f$ factorizes
	via $\bigoplus_{i\in I'} M_i$. From the adjunction
	described above we have for any subset $I'$ of
	$I$ the  commutative diagram	$$
	\xymatrix{
	A\mbox{-mod}\left( A\left[ \gamma \right],
	\bigoplus\limits_{i\in I'} M_i \right) \ar[d] \ar@/1ex/[r]^{\cong} & 
	\Vect_\Cat\left( R\left[ \gamma \right], \bigoplus\limits_{i\in I'}
	M_{i} \right) \ar[d]\ar[r]^{\cong} & \Vect\left( R, \bigoplus\limits_{i\in I'} \left(
	M_i \right)_{\gamma}\right)\ar[d]\\
	A\mbox{-mod} \left( A\left[ \gamma \right],
	\bigoplus\limits_{i\in I} M_i \right)\ar@/1ex/[r]^{\cong} & \Vect_\Cat\left( R[\gamma] ,
	\bigoplus\limits_{i\in I} M_i \right)\ar[r]^{\cong} & \Vect\left( R,
	\bigoplus\limits_{i\in I}\left( M_i \right)_{\gamma}
	\right) ,
	} 
	$$ whose
	horizontal arrows are isomorphisms
and
	vertical arrows are induced by the natural inclusion of
	$\bigoplus_{i\in I'}M_i$ into $\bigoplus_{i\in I} M_i$. 
	Let $f'\colon R\to \bigoplus_{i\in I} \left( M_i \right)_{\gamma}$ be
	the map that corresponds to $f$. Then $f'\left( 1 \right) \in
	\bigoplus_{i\in I'} \left( M_i \right)_\gamma$ for a finite subset
	$I'\subset I$. Thus $f'$ can be factorized via $\bigoplus_{i\in I'}
	\left( M_i \right)_{\gamma}$ and therefore $f$ can be factorized via
	$\bigoplus_{i\in I'}M_i $. 

	It is left to show that $X := \left\{\, A\left[ \gamma \right] \,\middle|\, \gamma\in \Cat \right\}$ is a generating set for
	$A$-mod. Let $M$ be an $A$-module. 
	For every $\gamma\in\Cat$ there is a free $R$-module $V_{\gamma}$ and a
	surjective homomorphism of $R$-modules $\psi_{\gamma}\colon
	V_{\gamma}\to M_{\gamma}$. Then $V = \left( V_{\gamma}
	\right)_{\gamma\in \Cat}$ is a free $\Cat$-graded $R$-module and
	$\psi = \left( \psi_{\gamma} \right)_{\gamma\in \Cat}$ is a surjection of
	$\Cat$-graded $R$-modules. Now $F_A\left( V \right)$ is a direct sum of
	objects from $X$, since $F_A$ commutes with direct sums. Moreover, the
	composition
	$$
	\xymatrix{F_A\left( V \right)\ar[r]^{F_A\left( \psi \right)}\ar[r] &
	F_A\left( M \right) \ar[r]^{\varepsilon} & M }
	$$
	is a surjecive homomorphism of $A$-modules. Therefore, $M$ is a quotient
	of a direct sum of objects from $X$, which shows that $X$ generates $A$-mod. 
\end{proof}
The proof of 
the criterion of semi-perfectness that extends~\cite[Theorme~8.1]{semi-perfect}
to the case of $\Cat$-graded algebras over an arbitrary commutative ring is
similar to one given in~\cite{semi-perfect} and we skip it: 

\begin{theorem}
	\label{mainsemi}
	Let $\Cat$ be a small category and $A$ a $\Cat$-graded $R$-algebra. The
	category $A$-mod is semi-perfect if and only if for every arrow $\gamma\in
	\Cat$ the algebra $A\left( \gamma:\gamma \right)$ is semi-perfect. 
\end{theorem}

We are now ready to prove the main theorem of the paper. 

\begin{proof}[Proof of Theorem~\ref{main0}]
	By	Propositions~\ref{abelian} and~\ref{genset} $A$-mod is a Grothendieck
	category and 
	$$
	\left\{\, A[\gamma] \,\middle|\, \gamma\mbox{ an arrow in }
	\Cat \right\}
	$$
	is a generating set of $A$-mod, whose elements are finitely generated
	projective $A$-modules.

	Suppose first that
	the rings $A\left( \gamma:\gamma \right)$ are left perfect for all
	arrows $\gamma\in \Cat$.
	Just like in the proof of  \cite[Theorem~8.1]{semi-perfect} there is an
	isomorphism of rings
	$A\mbox{-mod}(A\left[\gamma\right],A[\gamma])\cong
	A(\gamma:\gamma)^{op}$. Note that every left or right perfect ring is
	semi-perfect. Thus, $A\mbox{-mod}(A[\gamma],A[\gamma])$ is
	a semi-perfect ring. Hence $A[\gamma]$ is a semi-perfect object. 
	
	To prove that $A$-mod is perfect, by Proposition~\ref{new} it is enough to check that $\left\{\, A[\gamma]
	\,\middle|\, \gamma\mbox{ an arrow in }\Cat \right\}$ is a 
	$T$-nilpotent system.
	Let $f_k\colon A[\beta_k]\to A[\beta_{k+1}]$ be
	a sequence of $A$-homomorphisms such that $$f_k\in J\left( A[\beta_k],A[\beta_{k+1}] \right).$$ From the adjunction
	between $F_A$ and the forgetful functor  we have an isomorphism 
of $\Cat$-graded $R$-modules	
	\begin{align*}
		A\mbox{-mod}\left( A[\beta_k], A[\beta_{k+1}] \right)   & \cong
		\Vect_\Cat\left( R[\beta_{k}], A[\beta_{k+1}]\right) \\[2ex]& \cong \left( A[\beta_{k+1}]
		\right)_{\beta_k} \cong \bigoplus_{\alpha\beta_{k+1} = \beta_k}
		A_{\alpha} = A\left( \beta_k:\beta_{k+1} \right).
	\end{align*}
	There are two possibilities to consider:

	1) There are $k$ and $l$ such that
	$A\left( \beta_k:\beta_{k+l} \right)\cong 0$. Then $f_{k+l-1}\dots f_k=0$ and
	$$f_{k+l-1}\dots f_k\left( f_{k-1}\dots f_1\left( X
	\right) \right)=0$$ for any small subobject $X$ of $A[\beta_1]$. 

	2) We have $A\left( \beta_n:\beta_{n+m} \right)\not\cong0$ for all
	$n$, $m\in \N$. Then there is an arrow $\beta\in \Cat$ such that
	$\beta=\beta_n$ for infinitely many $n\in \N$. Let $n\left( k
	\right)$, $k\in \N$, be an increasing sequence of natural numbers such
	that $\beta_{n(k)} = \beta$ for all $k$. Define $g_k = f_{n(k+1)-1}
	\dots  f_{n(k)}$ and $g = f_{n(1)-1} \dots  f_1$. Then $g_k\in J\left( A\left[ \beta \right], A\left[ \beta \right] \right)$.
	Since $A\left( \beta:\beta \right)^{op}$ is  right perfect, the ideal
	$J\left( A\left( \beta:\beta \right)^{op} \right)$ is right
	$T$-nilpotent. Therefore there is $m\in\N$ such that $g_m\dots g_1=0$.
	Thus
	$$
	g_m\dots g_1g\left( X \right)= f_{n(m+1)-1} \dots  f_1\left(
	X \right) = 0
	$$
	for any small subobject $X$ of $A\left[ \beta_1 \right]$. 
	Thus $\left\{\, A\left[ \gamma \right] \,\middle|\, \gamma\in\Cat
	\right\}$ is a $T$-nilpotent system. 
	
	Suppose now that $A$-mod is a perfect category.  By Theorem~\ref{mainsemi} the rings $A\left( \beta:\beta
	\right)$ are semi-perfect. By definition of semi-perfect ring the
	quotient ring $\left.\raisebox{0.3ex}{$A\left( \beta:\beta \right)$}\middle/
	\raisebox{-0.3ex}{$J\left( A\left( \beta:\beta \right) \right)$}\right.$
	is semi-simple. Thus by Theorem~28.4(b) of~\cite{anderson-fuller} it is
	enough to show that the ideals $J\left( A\left( \beta:\beta \right)
	\right)$ are left $T$-nilpotent for every map $\beta\in \Cat$. We will
	show in fact that the ideals $J\left( A\left( \beta:\beta
	\right)^{op} \right)$ are right $T$-nilpotent. 
	
	Consider a sequence $f_k\in J\left( A\mbox{-mod}\left( A\left[ \beta \right],A\left[ \beta \right] \right)  \right)= J\left( A\left[ \beta \right],A\left[ \beta \right] \right)$, $k\in \N$. Then by Lemma~1 of
	\cite{harada} $\mathop{Im}f_1$ is a small subobject of $A\left[ \beta
	\right]$. By Proposition~\ref{harada3} the system $\left\{\, A\left[
	\beta \right]
	\,\middle|\, \beta\in\Cat \right\}$ is $T$-nilpotent. Therefore there is $m\in \N$ such that
	$$
	f_m\dots f_2\left( \mathop{Im}f_1 \right) = 0.
	$$
	Hence $f_m\dots f_2f_1 =0$. 
\end{proof}
\section{Examples}
\label{examples}
In this section we apply the main theorem to some classes of interesting
rings. 
\subsection{Algebras graded by a monoid}
Let $\Gamma$ be a monoid with unit $e$. We denote by
$(*, \Gamma)$ the category with one object $*$ and the set of morphisms
given by $\Gamma$. Recall that a  $\Gamma$-graded $R$-algebra  is an $R$-algebra
$A$ with a fixed direct sum decomposition into $R$-submodules $A_\gamma$, $\gamma\in
\Gamma$ such that $e_A\in A_e$ and  $A_{\alpha} A_{\beta} \subset
A_{\alpha\beta}$. Analogously, a $\Gamma$-graded module $M$ over a
$\Gamma$-graded $R$-algebra
$A$ is defined as an $A$-module with a direct sum decomposition $M=
\bigoplus_{\gamma\in \Gamma} M_\gamma$ of $R$-submodules such that $A_{\alpha} M_\beta \subset
M_{\alpha\beta}$.  Homomorphisms of $\Gamma$-graded $R$-algebras ($A$-modules) are 
homomorphisms of $R$-algebras ($A$-modules) that preserve the components of the direct sum
decomposition. 

It immediately follows that purely syntactical replacement of the sign
$\bigoplus_{\gamma\in \Gamma}$ by the brackets $(\ )_{\gamma\in \Gamma}$ gives
an
 equivalence between the category of $\Gamma$-graded algebras and
 the category $(*,\Gamma)$-graded algebras. By the same argument, if $A=\bigoplus_{\gamma \in
 \Gamma} A_{\gamma}$ is
 an $\Gamma$-graded $R$-algebra then the category $A$-gr of $\Gamma$-graded $A$-modules is
 equivalent to the category of $A'$-modules, where $A'$ is the $\left( *,\Gamma
 \right)$-graded algebra that corresponds to $A$. 

\subsubsection{Algebras graded by a group}
Now we assume that $\Gamma$ is a group and $A$ is a $\Gamma$-graded algebra.  Note that this is the most
widely studied case of graded algebras (the standard reference book on the
subject is~\cite{no}).

We denote by
$\Supp\left( A \right)$ the support of $A$, that is  the set of arrows $\gamma\in\Gamma$ such that
$A_{\gamma}\not \cong 0$.

\begin{proposition}
	\label{group}
	Let $\Gamma$ be a group and $A$ a $\Gamma$-graded $R$-algebra
	with finite support.
	 Then $A$-gr is 
	perfect if and only if $A_{e}$ is a left perfect ring. 
\end{proposition}
\begin{proof}Denote by $A'$ the  $\left( *,\Gamma \right)$-graded algebra that
	corresponds to $A$. 
	Then $A'$ and $\left( *,\Gamma \right)$ satisfy the conditions
	of Theorem~\ref{main0}. In fact, let $\left( \beta_k \right)_{k\in \N}$ be a
	sequence of elements in $\Gamma$ such that $A'\left(
	\beta_{k}:\beta_{k+l} \right)\not\cong 0$ for all $k$, $l\ge 1$. We have 
	for every $n\ge 2$
	$$
	0 \not\cong A'\left(\beta_1:\beta_n\right) = \bigoplus_{\alpha\beta_n=\beta_1}
	{A'}_{\alpha} = {A'}_{\beta_1\beta_n^{-1}} =
	A_{\beta_1\beta_n^{-1}}. 
	$$
	Thus $\beta_1\beta_n^{-1}$ lies in $\Supp(A)$. Since
	$\Supp\left( A \right)$ is finite at least one element repeats
	infinitely many times in the sequence $\left( \beta_1\beta_n^{-1}
	\right)_{n\in \N}$. As $\beta_n = \left( \beta_1\beta_n^{-1}
	\right)^{-1}\beta_1$ the same is true for $\left( \beta_n
	\right)_{n\in \N}$. 	
\end{proof}

This result was previously obtained in~\cite[Theorem~6(1,2)]{ndt} by a  different
technique. Note also that in~\cite{beattie1} it is proved that a $\Gamma$-graded
ring $A$ with finite support is left perfect as a usual ring
if and only if $A_e$ is left perfect. In fact, if $\Gamma$ is finite,  it was
also proved in~\cite{jensen} that a $\Gamma$-graded ring $A$ is left perfect as usual
ring both if and only if $A_e$ is left perfect, and  if and only if the category $A$-mod
is perfect. 
Chronologically the first results of this type are due to Renault~\cite{renault}
and Woods~\cite{woods} who gave a criterion for perfectness of group algebras
over a finite group. Their results were extended by Park in~\cite{park} to the
case of skew group rings. 

Let $A$ be a $\Gamma$-graded ring. The reader can find in~\cite{beattie2} a characterization of perfectness for the categories of
modules graded by $\Gamma$-sets. These categories do not fit in
the general framework of the present paper.

\subsubsection{Algebras graded by an ordered monoid}
 Recall that a poset $\left( S,\le \right)$ is called \emph{artinian} if every
 descending sequence $s_1 \ge s_2 \ge \dots $ of elements in $S$ stabilizes.

\begin{proposition}
	\label{gamma}
	Let $\Gamma$ be an artinian ordered monoid such that $e$ is the least
	element.  Then the category of left $\Gamma$-graded $A$-modules is
	perfect if and only if the ring $A_e$ is left perfect.
\end{proposition}
\begin{proof}
	Let $A'$ be the $\left( *,\Gamma \right)$-graded $R$-algebra that corresponds
	to $A$ under the equivalence described above. Let $\gamma\in \Gamma$. Then $\left\{\, \alpha \,\middle|\,
	\alpha\gamma=\gamma
	\right\} = \{e\}$. In fact, suppose
	$\alpha\gamma=\gamma$ and $\alpha\not = e$. Since $e$ is the least
	element of $\Gamma$ we have $\alpha>e$, and, as $\Gamma$ is an
	ordered monoid it follows that $\alpha\gamma>e\gamma=\gamma$,
	a contradiction. 
	Therefore for all $\gamma\in\Gamma$ $$(A')(\gamma:\gamma)= A'_e=
	A_e.$$  

	It is left to check that $\left( *,\Gamma \right)$ satisfies the
	condition of Theorem~\ref{main1}. 
	Suppose $\gamma_1$, $\gamma_2$, \dots is a sequence of elements in
	$\Gamma$ such that $\gamma_{k+1}$ is a right divisor of $\gamma_k$.
	Since $e$ is the least element of $\Gamma$ we get that $\gamma_k >
	\gamma_{k+1}$. Therefore $\gamma_1$, $\gamma_2$, \dots is a descending
	sequence and must stabilize as $\Gamma$ is artinian.
\end{proof} 
An example of a graded algebra in the conditions just described is the Kostant
form of the universal enveloping algebra of the complex Lie algebra of strictly upper
triangular matrices. In our work on Schur algebras~\cite{main}, we were led to
the construction of a minimal projective resolution of the trivial module of
this Kostant form. Although this module is obviously finitely generated it
can not to be said the same about the kernels of the projective covers which appear in the
resolution. It was this example that motivated the present paper. 

\begin{remark}
	In~\cite{eilenberg} Eilenberg gave a criterion for an $\N$-graded ring
	$A$
	to be perfect. Namely, Proposition~15 of~\cite{eilenberg} says that if $A_0$ is semiprimary then
	$A$ is graded perfect. Note that every semiprimary ring is perfect
	(p.318~\cite{anderson-fuller}) and therefore this result can be deduced
	from~Proposition~\ref{gamma} in this paper.
\end{remark}
Now we give an example which shows that the condition ``$\Gamma$ is artinian''
in Proposition~\ref{gamma} is essential. 

  Let $\Gamma = (\Z,+)$ and denote $\left( *,\Gamma
\right)$ by $\Cat$.
Given a field $\K$, define
a $\Cat$-graded $\K$-algebra $A$  by 
$$
A_k = \left\{ \begin{array}{cc}\K a_k & k\ge 0 \\ 0 & \mbox{otherwise}
\end{array} \right.,\  k\in \Z
$$
and multiplication $a_k a_l = a_{k+l}$. In fact, $A$ is just the polynomial
algebra in one variable considered as a $\Cat$-graded algebra. We define a
$\Cat$-graded
$A$-module $X$ by 
$$
X_k := \K x_k,\ k\in\Z
$$
and the action of $A$ on $X$ is given by $a_k x_l = x_{k+l}$. 

\begin{proposition}
	The module $X$ has no projective cover in $A$-mod.
\end{proposition}
\begin{proof}
	Suppose $\phi\colon P\twoheadrightarrow X$ is a projective cover of $X$.
	Then, by Theorem~5.1\cite{semi-perfect}, $P$ is a direct summand of the free module $F_A(X)$ and there is an
	idempotent $e\colon F_A(X)\to F_A(X)$ such that $fe=f$, where $f\colon
	F_A(X)\to X$ 
	is given by $f(a_k\otimes x_l)= x_{k+l}$.

	Note that for every $k$, the set $\left\{\, a_{k-l}\otimes x_l
	\,\middle|\, k\ge l
	\right\}$ is a basis of the vector space $F_A(X)_k$, so we can write	
	$e(a_0\otimes x_k) = \sum_{l\le k} 
	\lambda_{k,l}a_{k-l}\otimes x_{l}$, where $\lambda_{k,l}\in \K$. Now the coefficient of $a_0\otimes x_k$ in
$\sum_{l<
k}\lambda_{k,l}a_{k-l}e\left( a_0\otimes x_{l}\right)$ is zero. Therefore the
coefficient of $a_0\otimes x_k$ in
$e^2\left(a_0\otimes x_k\right)$ is $\lambda_{k,k}^2$. Since $e$ is an idempotent we get that
$\lambda_{k,k}^2 = \lambda_{k,k}$.

For every $k$ there are two possibilities: either $\lambda_{k,k}=1$ or
$\lambda_{k,k}=0$. Let $I\subset \Z$ be the set of $k$'s such that
$\lambda_{k,k}=1$.

We will show that the set $I$ contains infinitely many elements. Suppose
$k$ is the minimal element of $I$. Then either $e(a_0\otimes x_{k-1})=0$ or
$ e\left( a_0\otimes x_{k-1} \right)  = \lambda_{k-1,l}
a_{k-1-l}\otimes x_l+ \sum_{m<l}\lambda_{k-1,m}a_{k-1-m}\otimes x_m$, where
$\lambda_{k-1,l}\not=0$ and $l<k$.  The first alternative is impossible as
$$fe\left( a_0\otimes x_{k-1}  \right)= f\left( a_0\otimes x_{k-1} \right)=
x_{k-1}\not=0.$$ In the second case all the monomials different from
$\lambda_{k-1,l}a_{k-1-l}\otimes x_{l}$ in $e\left( a_0\otimes
x_{k-1} \right)$ are of the 
form $\lambda_{k-1,m}a_{k-1-m}\otimes x_m$ for $m<l$. 
Since the coefficient of $a_{k-1-l}\otimes x_l$ in $e\left(
\lambda_{k-1,m}a_{k-1-m}\otimes x_m \right)$ is zero, and $e\left( a_0\otimes
x_{k-1}
\right) = e^2\left( a_0\otimes x_{k-1} \right)$, it follows that the coefficient of
$a_{k-1-l}\otimes x_l$ in $e\left( a_{k-1-l}\otimes x_l \right)$ is one, or in
other words, that $l\in I$. This gives a contradiction between assumptions that $k$ is the
minimal element of $I$ and $l\le k-1<k$. 

%

Let us fix $k,l\in I$, $l<k$. Denote $a_0\otimes x_k - a_{k-l}\otimes x_l$ by
$v$. We have $e(v) = a_0\otimes x_k + \dots$, where all
other summands are of the form $\mu a_{k-m}\otimes x_m$, $m<k$. This shows that
$e(v) \not=0$. Moreover, $fe(v)= f(v) = x_k - x_k=0$. Thus $e(v)\in \ker(f)\cap
P= \ker\left( \phi \right)$. Next we show that $\ker\left( \phi \right)$ is not a
small subobject of $P$. For this we will find an $A$-submodule $Q$ of $P$ such
that $Ae(v) + Q =P$ and $e\left( v \right)\not\in Q$.

Let
$$
B' := \left\{\, e(a_0\otimes x_m) \,\middle|\, m\not \in I,\ m<k \right\}\cup
\left\{\, e(a_0\otimes x_i)
\,\middle|\, i\in I, i\not=k \right\}
$$
and $B =B'\cup \left\{ e(v) \right\}$.
We will prove that $B$ generates $P$ as an $A$-module. For this we have only to show that for
every $n>k$, $n\not\in I$ the element $e\left( a_0\otimes x_n \right)$ of
$P$ is in the $A$-linear span of $B$. 
We have
\begin{equation}
	\begin{split}
	e\left( a_0\otimes x_n \right)& = e^2\left( a_0\otimes x_n \right)\\& = 
	\lambda_{n,s} a_{n-s} e\left( a_0\otimes x_s \right)  + \sum_{t<s}
	\lambda_{n,t} a_{n-t} e\left( a_0\otimes x_t \right),
\end{split}
		\label{formula}
\end{equation}
	where $\lambda_{n,s}\not=0$. 
If $s=k$ then we can rewrite the above sum in the form
$$
e\left( a_0\otimes x_n \right) =
\lambda_{n,k} a_{n-k} e\left( v\right) + \sum_{t<k}
\mu_{n,t} a_{n-t} e\left( a_0\otimes x_t \right),
$$
where $\mu_{n,l}= \lambda_{n,l} + \lambda_{n,k}$ and $\mu_{n,t} =
\lambda_{n,t}$ for $t\not=l$. Thus $e\left( a_0\otimes x_n \right)$ belongs to
the $A$-linear span of $B$. 
If $s\not=k$, then $e\left( a_0\otimes x_s \right)\in B$ as $s\in I$. 
Now, for each $t<s$ in (\ref{formula}), either $t$  belongs to $I$, or $e\left( a_0\otimes
x_t 
\right)$ can be written as an~$A$-linear combination of elements $e\left(
a_0\otimes x_r 
\right)$ with $r<t$. So we keep applying $e$ to each of these until we get
only
$e\left( a_0\otimes x_r \right)$ with either $r\in I$ or $r\le k$ (note that we
are left with a finite number of indices to deal, since we are only concerned
with those $k<t\le s$).  We conclude then that $e\left( a_0\otimes x_n \right)$
belongs to the $A$-linear span of $B$.

Let us denote by $Q$ the $A$-linear span of $B'$. We will show that the
element  $e\left(
v \right)$ of $P_k$  is not in $Q_k$. Every element $w$ of $Q_k$  can be written
in the form
$$
w= \sum_{t\le k} \tau_t a_{k-t} e\left( a_0\otimes x_t \right),
$$
where the sum is over $t\not \in I$, $t<k$ and $t\in I$, $t\not=k$. Thus in fact
$$
w= \sum_{t< k} \tau_t a_{k-t} e\left( a_0\otimes x_t \right).
$$
Now for $t<k$ the coefficient of $a_0\otimes x_k$ in every $a_{k-t}e\left( 
a_0\otimes x_t 
\right)$ is zero. On the other hand the coefficient of $a_0\otimes x_k$ in
$e\left( v \right)$ is $1$. Thus it is impossible that $e\left( v \right) =w$. 
\end{proof}
\subsection{Poset-graded algebras}
Let $(\Lambda,\le)$ be a poset. 
Denote by $\widetilde{\Lambda}$ the
category with the set of objects $\Lambda$ and exactly one morphism
$\overline{\mu\lambda}$
from $\lambda$ to $\mu$ for $\mu\ge \lambda$.  

Let $A$ be a $R$-algebra.  Suppose  there is an orthogonal decomposition of $e\in A$
$$
e=\sum_{\lambda\in \Lambda} e_\lambda, \ e_\lambda e_\mu=\delta_{\lambda\mu}
e_\lambda
$$ 
such that $$e_\mu Ae_\lambda\not\cong  0 \Rightarrow \mu\ge \lambda.$$ 
In this case we can define a $\widetilde{\Lambda}$-graded $R$-algebra
$\widetilde{A}$ by 
$$
\widetilde{A}_{\overline{\mu\lambda}} :=\left\{\, \left\llbracket a \right\rrbracket
\,\middle|\,a\in e_\mu A e_\lambda
\right\} $$
with the $R$-module structure inherited from $e_\mu Ae_\lambda$ via the
bijection
$\left\llbracket a \right\rrbracket\mapsto a$. We define multiplication on
$\widetilde{A}$ by $$
\left\llbracket x \right\rrbracket \left\llbracket y \right\rrbracket :=
\left\llbracket xy \right\rrbracket,\ x\in e_\nu Ae_\mu,\ y\in e_\mu A
e_\lambda.
$$
Then $\left\llbracket e_\lambda \right\rrbracket\in
\widetilde{A}_{\lambda\lambda}$ are local units in $\widetilde{A}$.

An example which illustrates this  situation  is provided by the Schur algebra
$S^+\left( n,r \right)$ for the upper Borel subgroup of the general linear group
$\GL_n$, where as a poset we take the set  of all compositions of $r$ into at most $n$ parts with the dominance
order (see~\cite{green} and \cite{aps}).

\begin{proposition}
	Let $\Lambda$, $A$ and $\widetilde{A}$ be as above. Suppose that for
	every $\lambda$, $\mu\in\Lambda$, $\mu\ge \lambda$ the interval
	$[\lambda,\mu]$ is an artinian poset.  Then the
	category $\widetilde{A}$-mod is perfect if and only if each ring
	$e_\lambda A e_\lambda$ is left perfect. 
\end{proposition} 
\begin{proof}
	 Note first that
	 for every $\mu\ge
	\lambda$
$$
\left\{\, \overline{\nu\mu} \,\middle|\, \overline{\nu\mu} \overline{\mu\lambda}=
\overline{\mu\lambda} \right\} = \left\{ \overline{\mu\mu} \right\}.
$$
Therefore $\widetilde{A}\left(\overline{\mu\lambda}:\overline{\mu\lambda}
\right) = \widetilde{A}_{\overline{\mu\mu}} \cong e_\mu A e_\mu $. 

It is left to check that the category $\widetilde{\Lambda}$ satisfies the
condition of Theorem~\ref{main1}. 	
	Suppose $\alpha_1$, $\alpha_2$, \dots is a sequence of maps in
	$\widetilde{\Lambda}$ such that $\alpha_{k+1}$ is a right divisor
	of $\alpha_k$. Then there are $\lambda$, $\mu_1$, $\mu_2$, \dots in
	$\Lambda$ such that $\alpha_k = \overline{\mu_k\lambda}$ and
	$\mu_{k+1}\le\mu_k$. Since $[\lambda, \mu_1]$ is artinian and every
	$\mu_k$ lies in this interval, we get that $\mu_1>\mu_2>\dots$
	stabilizes. Therefore $\alpha_1$, $\alpha_2$, \dots stabilizes as
	well. 
\end{proof} 

\section{Appendix}
	As we mentioned before, in this appendix we prove that the radical of an
	abelian category is unique and characterize it. 
\begin{proposition}
	\label{rad}
	If $\cat$ is an abelian category then there is a unique radical in
	$\cat$. 
\end{proposition}
\begin{proof}
	We use the notation introduced immediately after Definition~\ref{dar}. 
	Let $I$ be a
	radical of $\cat$. 
	By Proposition~\ref{radical} $I(A,B) = \pi_B I(A\oplus B, A\oplus B) i_A = \pi_B J\left(
\cat(A\oplus B, A\oplus B) \right)i_A$. This shows that a  radical is unique if
it exists.

Next we show the existence of a radical in $\cat$. 
We will identify the ring $\cat(A\oplus B, A\oplus B)$ with the matrix ring 
$$
\left(
\begin{array}{cc}
	\cat(A,A) & \cat(B,A) \\
	\cat(A,B) & \cat (B,B)
\end{array}
\right)
$$
via the structure maps of the direct sum.
Define $J(A,B)$ by 
\begin{align*}
 J\left( A,B \right) =\left\{\, f \,\middle|\, \left( 
\begin{array}{cc}
	0 & 0 \\
	f & 0 
\end{array}
\right)\in J\left( \begin{array}{cc} \cat(A,A) & \cat(B,A) \\ \cat(A,B) &
	\cat(B,B)\end{array} \right) \right\}.
\end{align*}
First we show that $J$ is an ideal in $\cat$. 
Let $C$ be an object in $\cat$. Denote by $E$ the idempotent 
$$
\left( 
\begin{array}{ccc}
	1_A & 0 & 0 \\
	0 & 1_B & 0 \\
	0 & 0 & 0 
\end{array}
\right)
$$
in $\cat(A\oplus B \oplus C, A\oplus B\oplus C)$. By
Proposition~5.13~\cite{curtis-reiner} we have an isomorphism $J\left(
\cat(A\oplus B, A\oplus B) 
\right) \cong E J\left( \cat(A\oplus B\oplus C, A\oplus B\oplus C) \right)E$.
Therefore
$$
J(A,B) = \left\{\, f\colon A\to B  \,\middle|\,  
\left( \begin{array}{ccc} 0 & 0 & 0 \\ f & 0 & 0 \\ 0 & 0 & 0 \end{array}
	\right) \in J\left( 
	\begin{array}{ccc}
\cat(A,A) & \cat(B,A) & \cat( C,A ) \\
\cat(A,B) & \cat(B,B) & \cat(C,B) \\
\cat(A,C) & \cat(B,C) & \cat(C,C)
	\end{array}
	\right)
\right\}.
$$
 Let $g\colon B\to C$. Then 
$$
\left( 
\begin{array}{ccc}
0 & 0 & 0 \\
0 & 0 & 0 \\
0 & g & 0 
\end{array}
\right)
\left( 
\begin{array}{ccc}
0 & 0 & 0\\
f & 0 & 0 \\
0 & 0 & 0 
\end{array}
\right) = \left( 
\begin{array}{ccc}
0 & 0 & 0 \\
0 & 0 & 0 \\
gf & 0 & 0 
\end{array}
\right)\in J\left( 
\begin{array}{ccc}
\cat(A,A) & \cat(B,A) & \cat(C,A ) \\
\cat(A,B) & \cat(B,B) & \cat(C,B) \\
\cat(A,C) & \cat(B,C) & \cat(C,C)
\end{array}
\right).
$$
Switching the roles of $B$ and $C$ in the above considerations we obtain
$$
J(A,C) = \left\{\, h\colon A\to C  \,\middle|\,  
\left( \begin{array}{ccc} 0 & 0 & 0 \\ 0 & 0 & 0 \\ h & 0 & 0 \end{array}
	\right) \in J\left( 
	\begin{array}{ccc}
\cat(A,A) & \cat(B,A) & \cat( C,A ) \\
\cat(A,B) & \cat(B,B) & \cat(C,B) \\
\cat(A,C) & \cat(B,C) & \cat(C,C)
	\end{array}
	\right)
\right\}
$$
and therefore $gf\in J(A,C)$. This shows that $J$ is a left ideal of $\cat$. That
$J$ is a right ideal can be shown analogously. 

Now we have to check that $J(A,A) = J\left( \cat(A,A) \right)$. By definition we
have
$$
J(A,A) = \left\{\, f\colon A\to A \,\middle|\,  \left( 
\begin{array}{cc}
	0 & 0 \\
	f & 0 
\end{array}
\right) \in J\left( 
\begin{array}{cc}
\cat(A,A) & \cat(A,A) \\
\cat(A,A) & \cat(A,A)
\end{array}
\right) \right\}.
$$
Let $f\in J(A,A)$. Then 
$$
\left(\begin{array}{cc}
	f & 0 \\
	0 & 0 
\end{array}\right)=
\left(\begin{array}{cc}
0 & 1_A \\
0 & 0 
\end{array}\right)
\left(\begin{array}{cc}
	0 & 0 \\
	f & 0 
\end{array}
\right) \in J\left( 
\begin{array}{cc}
\cat(A,A) & \cat(A,A) \\
\cat(A,A) & \cat(A,A)
\end{array}
\right) 
$$
since $J\left( \cat(A\oplus A, A\oplus A) \right)$ is an ideal of
$\cat(A\oplus A, A\oplus A)$. As 
$$
\left(\begin{array}{cc}
f & 0 \\
0 & 0 
\end{array}\right)= 
e\left(\begin{array}{cc}
f & 0 \\
0 & 0 
\end{array}\right)e,
$$
where $e = \left( \begin{array}{cc}1_A & 0 \\ 0 & 0 \end{array} \right)$, we
	obtain by Proposition~5.13~\cite{curtis-reiner}
	\begin{align*}
	\left( 
	\begin{array}{cc}
f & 0 \\
0 & 0
	\end{array}
	\right)\in e J\left( 
	\begin{array}{cc}
\cat(A,A) & \cat(A,A) \\
\cat(A,A) & \cat(A,A) 
	\end{array}
	\right)e & = J\left( e\left( 
	\begin{array}{cc}
\cat(A,A) & \cat(A,A) \\
\cat(A,A) & \cat(A,A) 
	\end{array}
	\right)e \right)\\[2ex] &   =
	J\left( 
	\begin{array}{cc}
	\cat(A,A) & 0 \\
	0 & 0 
\end{array}
	\right).
\end{align*}
Therefore $f\in J(\cat(A,A))$ and $J(A,A) \subset J\left( \cat(A,A) \right)$. 

Now suppose that $f\in J(\cat(A,A))$. Then 
$$
	\left( 
	\begin{array}{cc}
f & 0 \\
0 & 0
	\end{array}
	\right)\in J\left( 
	\begin{array}{cc}
\cat(A,A) & \cat(A,A) \\
\cat(A,A) & \cat(A,A) 
	\end{array}
	\right)
$$
 and 
 $$
 \left(\begin{array}{cc}
	0 & 0 \\
	f & 0 
\end{array}\right)=
\left(\begin{array}{cc}
0 & 0 \\
1_A & 0 
\end{array}\right)
\left(\begin{array}{cc}
	f & 0 \\
	0 & 0 
\end{array}
\right) \in J\left( 
\begin{array}{cc}
\cat(A,A) & \cat(A,A) \\
\cat(A,A) & \cat(A,A)
\end{array}
\right) 
 $$
  since $J\left( \cat(A\oplus A, A\oplus A) \right)$ is an ideal of
  $\cat(A\oplus A, A\oplus A)$. Thus $f\in J(A,A)$ and $J\left(
  \cat(A,A) 
  \right) \subset J(A,A)$. 

\end{proof}

\section{Acknowledgment}
The authors would like to thank the anonymous referee for the valuable suggestions
that considerably improved the paper.
\bibliography{perfectring_arx}

\providecommand{\bysame}{\leavevmode\hbox to3em{\hrulefill}\thinspace}
\providecommand{\MR}{\relax\ifhmode\unskip\space\fi MR }
\providecommand{\MRhref}[2]{%
  \href{http://www.ams.org/mathscinet-getitem?mr=#1}{#2}
}
\providecommand{\href}[2]{#2}
\begin{thebibliography}{10}

\bibitem{anderson-fuller}
Frank~W. Anderson and Kent~R. Fuller, \emph{Rings and categories of modules},
  Springer-Verlag, New York, 1974, Graduate Texts in Mathematics, Vol. 13.
  \MR{MR0417223 (54 \#5281)}

\bibitem{beattie2}
M.~Beattie and S.~D{\u{a}}sc{\u{a}}lescu, \emph{Categories of modules graded by
  {$G$}-sets. {A}pplications}, J. Pure Appl. Algebra \textbf{107} (1996),
  no.~2-3, 129--139, Contact Franco-Belge en Alg{\`e}bre (Diepenbeek, 1993).
  \MR{MR1383168 (97a:16080)}

\bibitem{beattie1}
Margaret Beattie and Eric Jespers, \emph{On perfect graded rings}, Comm.
  Algebra \textbf{19} (1991), no.~8, 2363--2371. \MR{MR1123129 (93b:16076)}

\bibitem{bourbaki}
Nicolas Bourbaki, \emph{Elements of mathematics. {T}heory of sets}, Translated
  from the French, Hermann, Publishers in Arts and Science, Paris, 1968.
  \MR{MR0237342 (38 \#5631)}

\bibitem{curtis-reiner}
Charles~W. Curtis and Irving Reiner, \emph{Methods of representation theory
  with applications to finite groups and orders {V}ol. {I}}, Wiley Classics
  Library, John Wiley \& Sons Inc., New York, 1990, Reprint of the 1981
  original, A Wiley-Interscience Publication. \MR{MR1038525 (90k:20001)}

\bibitem{eilenberg}
S.~Eilenberg, \emph{Homological dimension and syzygies}, Annals of Mathematics
  \textbf{64} (1956), no.~2, 328--336.

\bibitem{green}
J.~A. Green, \emph{On certain subalgebras of the {S}chur algebra}, J. Algebra
  \textbf{131} (1990), no.~1, 265--280.

\bibitem{harada}
Manabu Harada, \emph{Perfect categories. {I}}, Osaka J. Math. \textbf{10}
  (1973), 329--341. \MR{MR0367017 (51 \#3262a)}

\bibitem{jensen}
Anders Jensen and S{\o}ren J{\o}ndrup, \emph{Smash products, group actions and
  group graded rings}, Math. Scand. \textbf{68} (1991), no.~2, 161--170.
  \MR{MR1129585 (93a:16026)}

\bibitem{mitchell}
Barry Mitchell, \emph{Theory of categories}, Academic Press, New York and
  London, 1965.

\bibitem{ndt}
C.~N{\u{a}}st{\u{a}}sescu, L.~D{\u{a}}u{\c{s}}, and B.~Torrecillas,
  \emph{Graded semiartinian rings: graded perfect rings}, Comm. Algebra
  \textbf{31} (2003), no.~9, 4425--4531. \MR{MR1995550 (2004g:16047)}

\bibitem{no}
Constantin N{\u{a}}st{\u{a}}sescu and Freddy Van~Oystaeyen, \emph{Methods of
  graded rings}, Lecture Notes in Mathematics, vol. 1836, Springer-Verlag,
  Berlin, 2004. \MR{MR2046303 (2005d:16075)}

\bibitem{park}
Jae~Keol Park, \emph{Artinian skew group rings}, Proc. Amer. Math. Soc.
  \textbf{75} (1979), no.~1, 1--7. \MR{MR529201 (80d:16009)}

\bibitem{renault}
Guy Renault, \emph{Sur les anneaux de groupes}, C. R. Acad. Sci. Paris S\'er.
  A-B \textbf{273} (1971), A84--A87. \MR{MR0288189 (44 \#5387)}

\bibitem{sai}
Yousin Sai, \emph{On regular categories}, Osaka J. Math. \textbf{7} (1970),
  301--306. \MR{MR0276300 (43 \#2047)}

\bibitem{aps}
A.-P. Santana, \emph{The {S}chur algebra {$S(B\sp +)$} and projective
  resolutions of {W}eyl modules}, J. Algebra \textbf{161} (1993), no.~2,
  480--504. \MR{MR1247368 (95a:20046)}

\bibitem{main}
Ana~Paula Santana and Ivan Yudin, \emph{The {K}ostant form of
  {$\mathfrak{U}(sl_n^+)$} and the {B}orel sublagebra of the {S}chur algebra
  {$S(n,r)$}}, in preparation.

\bibitem{woods}
S.~M. Woods, \emph{On perfect group rings}, Proc. Amer. Math. Soc. \textbf{27},
  49--52.

\bibitem{semi-perfect}
Ivan Yudin, \emph{Semi-perfect category-graded algebras}, Comm. in Algebra
  (2010).

\end{thebibliography}
\bibliographystyle{amsplain}

\end{document}